\documentclass[11pt]{article}

\usepackage[margin=1in]{geometry}
\usepackage{enumerate, amsmath, amssymb,latexsym}
\usepackage{amsthm, color, graphicx, hyperref, amscd}
\usepackage{xcolor}
\usepackage[normalem]{ulem}
\newcommand\nc{\newcommand}
\nc\DMO{\DeclareMathOperator}
\nc\ignore[1]{}     

\DMO{\Tr}{Tr}        \DMO{\N}{N}    \DMO{\ind}{ind}
\DMO{\tr}{Tr}                              \DMO{\PG}{{PG}}
\DMO{\AG}{{AG}}                        \DMO{\GL}{{GL}}
\DMO{\PGL}{{PGL}}                     \DMO{\PGaL}{P\Gamma L}
\DMO{\Gal}{Gal}                          \DMO{\GaL}{\Gamma L}

\nc{\uE}{{\textup E}}                              \nc{\uL}{{\textup L}}
\nc{\uS}{{\textup S}}                              \nc{\Zp}{{\mathbb Z}_p}
\nc{\cR}{{\mathcal R}}                            \nc{\cA}{{\mathcal A}}
\nc{\cB}{{\mathcal B}}                            \nc{\cT}{{\mathcal T}}
\nc{\cC}{{\mathcal C}}                            \nc{\cK}{{\mathcal K}}
\nc{\cN}{{\mathcal N}}                           \nc{\cM}{{\mathcal M}}
\nc{\cG}{{\mathcal G}}                           \nc{\ts}{{\theta_s}}
\nc{\cE}{{\mathcal E}}                            \nc{\cF}{{\mathcal F}}
\nc{\cP}{{\mathcal P}}                            \nc{\cH}{{\mathcal H}}
\nc{\cD}{{\mathcal D}}                           \nc{\cU}{{\mathcal U}}
\nc{\cL}{{\mathcal L}}                            \nc{\Lb}{{\mathbb L}}
\nc{\wa}{{\widehat{\alpha}}}                 \nc{\wg}{{\widehat{\gamma}}}
\nc{\cJ}{{\mathcal J}}                            \nc{\cQ}{{\mathcal Q}}
\nc{\cS}{{\mathcal S}}                           \nc{\cX}{{\mathcal X}}
\nc{\F}{{\mathbb F}}                             \nc{\Fq}{\F_q}
\nc{\Fqt}{\F_{q^t}}                               \nc{\K}{{\mathbb K}}
\nc{\GF}{\hbox{{\rm GF}}}                    \nc{\st}{{\cal S}_t}
\nc{\sett}[2]{\{#1\mid #2\}}
\nc{\so}{{\cal S}}                                  \nc{\sk}{{\cal S}_2}
\nc{\ad}{E}                                            \nc{\lmb}{\lambda}
\nc{\la}{\langle}                                     \nc{\ra}{\rangle}

\nc{\ZG}{\mathcal{Z}(\mathrm{\Gamma L})}
\nc{\G}{\mathrm{\Gamma L}}                \nc{\ep}{\epsilon}
\nc{\vv}{\mathbf v}                                \nc{\hs}{{\hat{\sigma}}}

\newtheorem{theorem}{Theorem}[section]

\newtheorem{lemma}[theorem]{Lemma}
\newtheorem{corollary}[theorem]{Corollary}
\newtheorem{definition}[theorem]{Definition}
\newtheorem{notation}[theorem]{Notation}
\newtheorem{proposition}[theorem]{Proposition}
\newtheorem{result}[theorem]{Result}
\newtheorem{example}[theorem]{Example}

\newtheorem{property}[theorem]{Property}

\newtheorem{remark}[theorem]{Remark}
\newtheorem{ass}[theorem]{Assumption}
\nc{\red}[2]{\textcolor[rgb]{1,0.1,0.1}{#1}}
\nc{\green}[2]{\textcolor[rgb]{0.2,0.9,0.1}{#1}}

\nc{\corr}[2]{\textcolor[rgb]{1,0.1,0.1}{#1}}

\newcommand{\zsst}[1]{%
    \bgroup
    \markoverwith{\textcolor{magenta}{\rule[0.5ex]{2pt}{1.2pt}}}%
    \ULon{#1}%
    \egroup
}

\makeatletter
\renewcommand*{\@fnsymbol}[1]{\@arabic{#1}}
\newcommand{\silentthanks}[1]{\stepcounter{footnote}\protected@xdef\@thanks{\@thanks\protect\footnotetext[\the\c@footnote]{#1}}}
\makeatother

\title{Two-coloring the points of a projective plane}
\author{Aart Blokhuis\thanks{Department of Mathematics and Computer Science,
   Eindhoven University of Technology, P.O.\ Box 513, 5600 MB Eindhoven,
   The Netherlands. \href{mailto:a.blokhuis@tue.nl}{a.blokhuis@tue.nl}}
 \and
 \'Ad\'am Mark\'o\thanks{Institute of Mathematics, ELTE E\"otv\"os Lor\'and
   University, P\'azm\'any P\'eter s\'et\'any 1/C, H-1117 Budapest, Hungary.
   }
 \and
 Tam\'as Sz\H onyi\textsuperscript{2,3,4,5}%
 \silentthanks{HUN-REN, Alfréd Rényi Institute of Mathematics, Re\'altanoda utca 13-15, 1053, Budapest, Hungary}
   \silentthanks{HUN-REN-ELTE Geometric and Algebraic Combinatorics Research Group,
   P\'azm\'any P\'eter s\'et\'any 1/C, H-1117 Budapest, Hungary.
   \href{mailto:tamas.szonyi@ttk.elte.hu}{tamas.szonyi@ttk.elte.hu}}%
   \silentthanks{University of Primorska FAMNIT, Glagolja\v ska ulica 8,
   6000 Koper, Slovenia.}
 \and
Zsuzsa Weiner\textsuperscript{4,6}%
   \silentthanks{Prezi.com, Budapest, Hungary.
   \href{mailto:zsuzsa.weiner@gmail.com}{zsuzsa.weiner@gmail.com}}
}
\date{}

\begin{document}
	\maketitle

\begin{abstract}
Let us color the points of a projective plane of odd order $n$ with two colors, say red and blue. Erd\H os asked the question of determining the maximum number of balanced lines (that have the same number of red and blue points). The answer is $n^2$, and in the corresponding trivial coloring the unbalanced lines go through a point. In this paper, we consider the Galois plane $\PG(2,q)$ and determine the number of unbalanced lines for a non-trivial coloring. The problem can also be generalized to $(r,b)$-colored lines (see Definition \ref{def:rb-colored}) and in many cases we extend the results to the generalized problem. Our strongest result for Galois planes of prime order $p$ describes all 2-colorings with at most 
$\max \{3p+1,4p-22\}$ not $(r,b)$-colored lines (see Theorem \ref{thm:prime max}). Finally, in case of the generalized problem, we allow even more not $(r,b)$-colored lines if $r$ is small and the order is prime (for more details, see Theorem \ref{thm:mainprime}).
\end{abstract}

\section{Introduction}

Determining the chromatic number of a projective plane of order $n$ is an easy exercise. It is $3$ for the Fano plane and $2$ for $n\geq 3$. In spite of this, there are several interesting problems for colorings of projective planes. For example, one can impose extra conditions on color classes (one may require that they are arcs, or blocking sets, see Csima, F\"uredi \cite{CsF}, or \cite{BMPSz}), or can consider other variants of the chromatic number (such variants were introduced by Voloshin \cite{V}, and for projective planes the upper chromatic number was considered first by Bacs\'o and Tuza \cite{BT}, then in \cite{BHSz}). One may also consider balanced colorings in which the size of color classes differs by at most one (the balanced upper chromatic number was introduced by Araujo-Pardo, Kiss, Montejano \cite{A-PKM}, and considered later in \cite{BBMNSz}). There are probably many other interesting problems for coloring; for example, Alon and F\"uredi \cite{AF} consider colorings in which the distribution of colors is unique for the lines.

We will mostly be interested in 2-colorings and the colors we call red and blue. The classical result by Spencer \cite{S} about the discrepancy of projective planes gives a probabilistic proof of the existence of a 2-coloring so that for every line $\ell$ the difference of the number of red and blue points on $\ell$ is at most $K\sqrt n$ (where $K$ is an absolute constant and $n$ is the order of the plane). 

Probably motivated by the above result by Spencer, Erd\H os asked in the eighties the following problem:
{\em What is the maximum number of lines that have the same number of red and blue points?} Let us call such a line {\em balanced}. 

It is easy to see that there are at least $q+1$ unbalanced lines in $\PG(2,q)$. Indeed, as $q$ lines cannot cover all points, there exists a point through which only balanced lines pass. If this point is red, then the number of red points would be $(q^2+1)/2$. Similarly, on an unbalanced line there must be a point through which no other unbalanced line passes. Counting the number of red points on the lines through such a point yields a contradiction, showing that there are indeed at least $q+1$ unbalanced lines. A slight refinement of this simple argument also shows that when there are exactly $q+1$ unbalanced lines, then they must be concurrent. Moreover, half of these lines are entirely red, half are entirely blue (except the carrier of this pencil, which can be either color). This {\em trivial coloring} shows that the bound $q+1$ for the number of unbalanced lines is sharp, and essentially this is the only coloring with $q+1$ unbalanced lines.

We generalize the notion of balanced lines by introducing {\em $(r,b)$-colored} lines.

\begin{definition}
\label{def:rb-colored}
Color the points of $\PG(2, q)$ with red and blue and fix an integer $0 \leq r \leq q+1$ and let $b = q+1-r$. 
The line $\ell$ in $\PG(2, q)$ is {\em $(r,b)$-colored}, if the number of red points on $\ell$ is $r$, while the number of blue points is $b$. 
\end{definition}

For the remainder of this paper, we refer to $(r,b)$-colored lines as \emph{balanced lines} and all other lines as \emph{unbalanced lines}. Note that when $r = b = (q+1)/2$, this coincides with Erd\H os' original notion of balance. Also note that the role of the colors is symmetric, and so we may assume that $r\le b$. 

\begin{definition}
\label{def: trivial}
We call the following coloring trivial: take a pencil, color $r$ lines of the pencil entirely red (except for their common carrier, which can be any color), and color the remaining lines entirely blue.     
\end{definition}

The question we consider is to find the maximum number of $(r,b)$-colored (balanced) lines. Originally, as discussed above, Erd\H os asked this question for $r=\frac{q+1}{2}$. The third author heard about the problem from Jenő Lehel \cite{LJ}. Extending the argument above, one can see that there are exactly two exceptional cases when the number of unbalanced lines is smaller than $q+1$: when $r=0$, that is, all points are blue (no unbalanced lines), and when $r=1$, that is, one line is completely red while all other points are blue (exactly one unbalanced line). In all other cases, there are at least $q+1$ unbalanced lines.

The corresponding coloring in the case of equality is again the trivial coloring defined above. 

We shall focus on a stability version of this generalized Erd\H os question: characterize the case(s) when the number of $(r,b)$-colored lines is close to its maximum.

Note that the case $r = 0$ is trivial. If we color all points blue, then all lines are balanced.  In this case, the stability version of the problem given above would ask for the number of lines intersecting a point set (see also \cite{BB}).

Note that a balanced line contains $r \bmod p$ red points. Therefore, if we consider the set of red points $R$, then only the unbalanced lines can be not $(r \bmod p)$-secants of it (though not necessarily). Hence, if there are $\delta$ unbalanced lines, then the number of not $(r \bmod p)$-secants of $R$ is at most $\delta$. This allows us to apply the results of \cite{SzW} for sets having few not $(r \bmod p)$-secants.

In the present paper we shall work in the Desarguesian plane $\PG (2,q)$, $q = p^h$,  $p$ prime,  so for the order of the plane we will use $q$ and not $n$ as above. The corresponding affine plane $\AG(2, q)$ will occasionally be identified with $\F_{q^2}$ (see \cite{H}, Section 2.2). Our goal is to determine and classify 2-colorings that admit only a small number of unbalanced lines — in particular, to describe all such colorings when their number is at most $2q + 1$, and to analyze special cases when $q$ is prime.

\section{Outline of results and motivating examples}

Our main results provide complete classifications for colorings when there are few unbalanced lines. 


In Section \ref{sec: genq}, we deal with the general case $r\leq b$ in $\PG(2,q)$, $q=p^h$. Here we settle the problem under the additional assumption that there is a line for which the number of red points is not $r \pmod{p}$ (Assumption \ref{ass: non r mod p line}). This condition is needed when the unbalanced lines (i.e. not $(r,b)$-colored) do not form a dual blocking set (see Subsection \ref{subs: not dual bl set}). In this case we can describe the coloring completely if there are at most $2q+1$ unbalanced lines and $q>3$ (Theorem \ref{thm:not bl and non r mod p}). If, on the other hand,  the unbalanced lines form a dual blocking set (see Subsection \ref{subs: dual bl set}), then we describe the coloring completely if there are less than $2q+1$ lines (Theorem \ref{thm: unbalanced are bl set}). 

Section \ref{sec:orig} deals with the case of the original problem of Erd\H os ($r = b = (q+1)/2$). Using a counting argument (without Assumption \ref{ass: non r mod p line}), we prove that for a non-trivial coloring the number of unbalanced lines is at least $2q+1$ when $q>3$ (Theorem \ref{thm:erdos}). The argument also works for the nearby cases $|r-(q+1)/2|\leq (p-3)/2$: if $3\leq r\leq b$ and the unbalanced lines do not form a dual blocking set, then there are at least $2q+2$ unbalanced lines (Corollary \ref{cor:nearby}).

The next examples show that Assumption \ref{ass: non r mod p line} cannot be omitted: there exist colorings with fewer than 
$2q+1$ unbalanced lines which do not form a dual blocking set, and for which every line contains the same number of red points modulo $p$. Also one might think that in the Erd\H os case, the only possible example with $2q+1$ unbalanced lines is the trivial example with flipping the color of a point different from the carrier of the complete red lines. However, as Example \ref{ex: baer plus lines} shows with $t = (q-1)/2$,  this is not the case. 

\begin{example}
\label{ex:hyperoval}  
Let $q$ be even, and consider a set of $q+2$ lines in $\PG(2,q)$ such that no three of them are concurrent (a so-called \emph{dual hyperoval}). 
Color all points on these $q+2$ lines blue and all other points red.
Or color one of these $q+2$ lines blue, the remaining points on the other $q+1$ lines red, and all other points blue.

Since each point lies on either $0$ or $2$ of these $q+2$ lines, every line not belonging to the dual hyperoval contains exactly $r = q/2$ red points.
Thus these colorings have exactly $q+2$ unbalanced lines (namely, the lines of the dual hyperoval).
\end{example}

\begin{example}
\label{ex: hyperoval one point flipped}
In Example \ref{ex:hyperoval} change the color of a point on one of the lines of the dual hyperoval. This coloring will have exactly $2q+1$ unbalanced lines.
\end{example}

\begin{example}
\label{ex: baer plus lines} 
Let $q$ be a square, and let $B$ be a Baer subplane of $\PG(2,q)$. 
Choose a point $P \in B$, and let $t$ lines through $P$ be tangent to $B$ (i.e. meeting $B$ in $P$ only), where $1 \leq t \leq q-\sqrt q$. Color all points of $B \setminus \{P\}$ red, and also color red all points on these $t$ tangent lines except $P$. The color of $P$ can be either red or blue. All remaining points are colored blue.

A line tangent to $B$ and not passing through $P$ contains exactly $t+1$ red points. Any other line not passing through $P$ contains $\sqrt{q}+1+t$ red points, while the lines through $P$ contain either $0, 1, \sqrt q, \sqrt q + 1, q, q+1$ red points.

Let $r = t+1$ and $b = q - t$. Then:
\begin{itemize}
    \item if $t = \sqrt{q} - 1$ and $P$ is blue, then the coloring has exactly $2q - \sqrt{q}$ unbalanced lines;
    \item if $t = \sqrt{q}$ and $P$ is red, then the coloring has exactly $2q - \sqrt{q}$ unbalanced lines;
    \item if $t \neq \sqrt q -1, \sqrt{q}$, then the coloring has exactly $2q + 1$ unbalanced lines.
\end{itemize}
\end{example}

Section \ref{sec: prime} is devoted to the case $q=p$ prime. The non-$(r \bmod p)$ secants are typically just the unbalanced lines (except for $r = 1$), so it is not surprising that we have stronger results for $q = p$. First of all, there is an example due to Kiss and Somlai, who constructed a subset of size $p(p-1)/2$ which intersects each line with slope not $m = 0, 1$ in precisely $(p-1)/2$ points. If we color this set and the line at infinity red, then we get a non-trivial coloring in the Erd\H os case with $3p-2$ unbalanced lines.
In the prime case, $q = p \geq 19$, we can describe colorings with at most $\max\{3p+1, 4p-22\}$ unbalanced lines completely, see Theorem \ref{thm:prime max} in Subsection \ref{subs: st most max}. Finally, in Subsection \ref{subs: more unbalanced with restr}, we consider the general problem when $r$ is small in $\PG(2,p)$, $p$ prime. In this case the red points are the union of $r$ completely red lines (with some assumption on the number of points they intersect) and we may change the color of a small number of points (Theorem \ref{thm:mainprime}).

\section{Colorings with at most  \(2q + 1\) unbalanced lines in \(\PG (2,q)\)} 
\label{sec: genq}
In this section, we color the points of $\PG(2,q)$ red and blue, and we denote the set of red points by $R$, the set of blue points by $B$, and the set of unbalanced lines by $U$. 
Throughout this section, we will assume that  $|U| \leq  2q + 1$, unless stated otherwise and we will characterize the possible such colorings. We will distinguish between two cases: first, we consider the case where the unbalanced lines form a dual blocking set, and then we examine the case where there exists a point not covered by any of the unbalanced lines.

\subsection{Unbalanced lines form a dual blocking set}
\label{subs: dual bl set}
\begin{lemma}
\label{lem: balanced_line}
If $0< |U| \leq 2q + 1$ and $U$ covers all points of $\PG (2,q)$, then every balanced line contains a point through which there is a unique unbalanced line. \qed
\end{lemma}

First, we consider the case when the points through which a unique unbalanced line passes are all the same color. We will need the following result.

\begin{result}[Tallini Scafati {\cite{tallini1}}, {\cite{tallini2}}] \label{res:(1,m)}
If $q$ is a power of a prime, then any point set in a projective plane of order $q$ intersecting each line in either $1$ or $n$ points, where $0 \leq n \leq q+1$ and $n \neq 1$, is one of the following: a unique point, a line, a Baer subplane, or a unital.
\end{result}

\begin{lemma}
\label{lem: a point, Baer}
Assume that unbalanced lines cover all points of $\PG (2,q)$ and the points through which a unique unbalanced line passes are all the same color, say red. Suppose also that all unbalanced lines contain the same number of red points and the total number of unbalanced lines is less than $q\sqrt{q} + 1$. Then either there is a unique blue point or the blue points form a Baer subplane.

\end{lemma}

\begin{proof}
By assumption, the number of red points on a line can take on two possible values. Hence, the number of unbalanced lines through any two points of the same color must be the same. Since the points through which a unique unbalanced line passes are red, we conclude that exactly one unbalanced line passes through each red point. By Result \ref{res:(1,m)}, this means there is either only one unbalanced line, or $q+1$ concurrent ones, or they form a dual Baer subplane. Unbalanced lines form a dual blocking set and hence the first case cannot happen.
The lemma now follows from the fact that the points through which there passes a unique unbalanced line are exactly the red points.
\end{proof}

\begin{lemma}
\label{lem: line and a point}
Assume that all points through which a unique unbalanced line passes are the same color, say red. Suppose also that the unbalanced lines form a dual blocking set and there exist two unbalanced lines with different
number of red points on them. Assume that $|U| < 2q + 1$. Then, a line and a point disjoint from it are colored blue, and the rest of the points are red.
\end{lemma}

\begin{proof}
If $R$ denotes the set of red points, then on an unbalanced line containing a red point through which that line is the only unbalanced line, we see $r_0 := |R| - q(r-1)$ red points.  

Suppose that $u$ is an unbalanced line with $r_1 \neq r_0$ red points. Since the unbalanced lines form a dual blocking set and have size less than $2q+1$, by 
Sz\H onyi (\cite{szonyiblsets}, Proposition 3.1), the essential lines (i.e., the lines containing a point through which that line is the only unbalanced line) form a minimal dual blocking set. This implies that, through the points of a line $u$, apart from $u$ itself, we observe at least one more unbalanced line with $r_0$ red points.

The fact that $r_0 = |R| - q(r-1)$ yields that if through a red point on $u$ there is another unbalanced line with $r_0$ red points, then there must be a third unbalanced line through that point. If there are no blue points on $u$, then we already count at least $2(q+1) + 1$ unbalanced lines, which contradicts the assumption.

If every blue point on $u$ has at least three unbalanced lines passing through it, then again $|U| \geq 2(q+1) + 1$, a contradiction. Thus, there is a blue point $P$ on $u$ with only two unbalanced lines. Namely, $u$ and the other compulsory unbalanced line $e$ containing $r_0$ red points. Counting the total number of red points on the lines through $P$, we obtain $(q-1)r + r_0 + r_1$. However, we also know that there are $q(r-1) + r_0$ red points in total. This implies that $r_1 = 0$ and $r = q$ or $r_1=1$ and $r=q+1$.
In the latter case, there can be at most one blue point (otherwise there was at least $2q+1$ lines containing blue points, i.e. unbalanced), which is impossible since $u$ contains $q$ blue points. In the former case, all points not on $u$ or $e$ are red. Now, if we take another point $P'$ on $u$, then, as mentioned above, there is an unbalanced line through $P'$ with $r_0$ red points. Since all points outside of $u$ and $e$ are red, this means that $q-1 \leq r_0 \leq q$. However, a balanced line has exactly $q$ red points, so we must have $r_0 = q-1$, from which the proof follows. 
\end{proof}

Next, we will consider the case where there are two points of different colors, each having a unique unbalanced line passing through it.

\begin{lemma}
\label{lem:2_unique_unbalanced}
Let $P$ and $Q$ be two points of different colors. Assume that there is a unique unbalanced line through each of these points, then the line $PQ$ is balanced. Let $\ell_P$  and $\ell_Q$ be the unique unbalanced lines through $P$ and $Q$ respectively and $T:= \ell_P \cap \ell_Q$. Then the points of $\ell_P \setminus \{T\}$ are colored the same as $P$, and similarly, the points of $\ell_Q \setminus \{T\}$ are colored the same as $Q$.
\end{lemma}

\begin{proof}
Assume that $P$ is colored red.
If $PQ$ is unbalanced, 
counting the number of red points not on $PQ$ along the lines through $P$ would give $(r-1)q$
while through $Q$ would give
$rq$ which is a contradiction.
Let the number of red points on $\ell_P$ be $r_1$, and on $\ell_Q$ be $r_2$. By counting the number of red points on the lines through $P$, we get $q(r - 1) + r_1$. Similarly, by counting the red points on the lines through $Q$, we get $qr + r_2$. It follows that $r_1 - r_2 = q$ and hence we are done.
\end{proof}

\begin{proposition}
\label{prop: not the same color}
Assume that the points through which there is a unique unbalanced line are not all of the same color. Assume also that $|U|<2q+1$. Then, there exists a point $T$ such that the points, excluding $T$, on exactly $r$ lines through $T$ are colored red, and the points, excluding $T$, on the remaining $q+1-r$ lines through $T$ are colored blue. The point $T$ itself can be either red or blue.
\end{proposition}

\begin{proof}
 The assumption of the proposition and Lemma \ref{lem:2_unique_unbalanced} imply that there exists a point $P$ colored red, a point $Q$ colored blue, and a point $T$ such that the points of $TP \setminus \{T\}$ are red, and the points of $TQ \setminus \{T\}$ are all blue. Without loss of generality, we may assume that $T$ is colored blue. Hence, the line $TQ$ is completely blue. Lemma \ref{lem:2_unique_unbalanced} now shows that a unique unbalanced line through a blue point is completely blue and passes through $T$. Assume that there are $k$ such completely blue lines through $T$. Note that since all lines but $TP$ through $P$ are balanced, $r \neq q+1$.

First, assume that $k = 1$. Color every point on the line $TQ$ red. 
If $r$ was $q$, then since through $Q$ every line but $TQ$ was balanced, now the entire plane will be red and hence the proposition follows. 
If $r< q$ then in this new coloring, let us call lines with $r+1$ red points be balanced. The number of red points on every line but $TQ$ was increased by one, and so in this new coloring balanced lines are exactly those that were balanced earlier ($TQ$ remains unbalanced since $q+1 \neq r+1$). 
But now, the points with a unique unbalanced line through them are all red, the unique unbalanced line is completely red, and there are at least two such lines ($TQ$ and $TP$). A contradiction follows from 
Lemma \ref{lem: a point, Baer}
and from Lemma \ref{lem: line and a point} and hence if $k=1$ then $r=q$. 

Now suppose that $k > 1$. Note that this means that if a unique unbalanced line passes through a red point $P'$, then it must be the line $P'T$ and the points of this line are all red except $T$.
We show that each line through $T$ is unbalanced. On the contrary, suppose that there is a balanced line $\ell$ through $T$. From above this means that there is no point on $\ell$ through which a unique unbalanced line passes, otherwise $\ell$ would be unbalanced. (Note that through $T$ there pass at least three unbalanced lines). Since unbalanced lines cover all points, counting them through the points of $\ell$ we see at least $2(q+1)$ unbalanced lines; a contradiction.
Hence all lines through $T$ are unbalanced. Let $e$ be a line through $T$.
Now if through each point of $e\setminus\{T\}$, another unbalanced line passes different from $e$, then again in total we count at least $2q+1$ unbalanced lines. This means that each line through $T$ contains a point $S$ such that $ST$ is the unique unbalanced line through $S$ and hence the proposition follows. 
\end{proof}

We summarize the results above in the following theorem.

\begin{theorem}
\label{thm: unbalanced are bl set}
Color the points of the plane $\PG (2,q)$ either red or blue. A line is balanced if it contains $r$ red and $b$ blue points, where $0 \le r \leq b \le q+1$. Assume unbalanced lines form a dual blocking set and there are less than $2q+1$. Then, the red points form one of the following structures:

\begin{enumerate}[(1)]
    \item a point, 
    \item a line and a point not on the line, 
    \item $r$ concurrent lines, 
    \item $r$ concurrent lines minus their intersection, 
    \item a Baer subplane. 
\end{enumerate}
\end{theorem}

\begin{proof}
By Lemma \ref{lem: balanced_line}, there are points uniquely covered by unbalanced lines. If these points are of different colors, then by Proposition \ref{prop: not the same color}, we are in cases $(3)$, $(4)$. If they are of the same color, then Lemmas \ref{lem: a point, Baer} and \ref{lem: line and a point}, along with the assumption that $r \leq b$, lead to cases $(1)$, $(2)$, and $(5)$.
\end{proof}

\subsection{Unbalanced lines do not form a dual blocking set}
\label{subs: not dual bl set}
\subsubsection{Preliminaries for general $q$}
In this subsection, we assume that the unbalanced lines do not cover all points. By counting the red points through such a non-covered point, we immediately see that all such points must have the same color. Hence, we make the following assumption.

\begin{ass}
\label{ass: non-covered are blue}
Points that are not covered by unbalanced lines are blue, and thus the total number of red points is $(q+1)r$.  
\end{ass}

\begin{proposition}
\label{prop: structure}
Assume that points not covered by unbalanced lines are blue. Then, either a line is completely red with the points outside being blue, or no point is covered by exactly one unbalanced line.
\end{proposition}

\begin{proof}
Assume that $P$ is a point through which exactly one unbalanced line passes. Since points not covered by an unbalanced line are blue, by counting the red points on the unbalanced line through $P$, we see that $P$ must be red. Let $l$ be the unique unbalanced line through $P$ with $r_0$ red points. Counting the red points on the lines through $P$, we find that the total number of red points is $q(r-1)+r_0$, which is $(q+1)r$ (by Assumption \ref{ass: non-covered are blue}). Hence, we have $r_0 - r = q$, i.e., either $r_0 = q+1$ and $r=1$, or $r_0 = q$ and $r=0$. In the first case, looking around from $P$, we see that all points outside $l$ must be blue. The second case is not possible, as it implies that balanced lines are completely blue, but $P$ is red.
\end{proof}

Next, we will use the famous nuclei polynomial to show that there cannot be too many points not covered by the unbalanced lines. 
To that we identify the points of AG$(2,q)$ with the elements in $\F_{q^2}$ and so three distinct points $a$, $b$, $c$ are collinear if and only if
$(a-c)^{q-1} = (b-c)^{q-1}$ holds. This way the ideal
points (directions) correspond to $(q+1)$-th roots of unity. (For more details see \cite{H}, Section 2.2.)
We will need the following assumption.

\begin{ass}
\label{ass: non r mod p line}
There is a line for which the number of red points is {\bf not} $ r\pmod p $.  
\end{ass} 

The example coming from a dual hyperoval, mentioned in the Introduction, shows that the existence of lines with not $r$ mod $p$ secants is relevant. The second author showed in \cite{M} that such examples only occur for $q$ even if the number of unbalanced lines is at most $4q/3$. He also described the two non-trivial examples.

\begin{proposition}
\label{prop: uncovered points}
Assume that the unbalanced lines do not cover all points. Under Assumption \ref{ass: non r mod p line}, there are at most $q-1$ points not covered by an unbalanced line.
\end{proposition}
\begin{proof}
Let us choose the line at infinity as the line guaranteed by Assumption \ref{ass: non r mod p line}. Then, the number of red points on the remaining affine plane is non-zero modulo $p$.

Identify the points of AG$(2,q)$ with the elements in $\F_{q^2}$, and hence the affine red points $R'$ form a subset of $\F_{q^2}$. Let $a_1, \dots, a_k$ be the red points on the line at infinity, which can be viewed as the $(q+1)$-th roots of unity. Write the following polynomial for the red points:
\[ 
  F(X) =\sum_{a\in  R'} (X-a)^{q-1}+a_1+\ldots + a_k,
\]
The polynomial is not identically zero, since its leading coefficient is non-zero (it corresponds to the number of affine red points). When $X = x_0$ is a blue point not covered by an unbalanced line (Assumption \ref{ass: non-covered are blue}), $x_0$ is a root of $F(X)$. This is because all the lines through $x_0$ contain $r$ red points (including the point at infinity), so $F(x_0)$ is $r$ times the sum of all $(q+1)$-th roots of unity, which equals zero. This proves the proposition.
\end{proof}

By Proposition \ref{prop: structure}, if we are not in the case where a line is completely red with the points outside being blue, then the unbalanced line in the dual plane forms a set with no tangents, and at most $q-1$ lines are disjoint from it. Such sets can be characterized if we impose a suitable bound on their size.

\begin{lemma}
\label{lem: 2 lines}
Assume that $S$ is a set of size at most $2q+1$ in $\PG (2,q)$, with no tangents and at most $q-1$ lines disjoint from $S$. Then, $S$ is the symmetric difference of two lines, unless $q=4$, $|S|=9$, and $S$ is the complement of the union of three non-concurrent lines.
\end{lemma}
\begin{proof}
Pick any point $P$ not in $S$. Since $|S| \leq 2q+1$, there is a line through $P$ that is disjoint from $S$. In other words, the points outside $S$ are covered by lines disjoint from $S$. There are at most $q-1$ lines disjoint from $S$, and they can cover at most $(q-1)q+1$ points (equality holds only if there are $q-1$ passant lines and these lines are concurrent). Thus, $S$ must contain at least $2q$ points, and for $|S|=2q$, the set $S$ is the symmetric difference of the two remaining lines passing through the common points of the $q-1$ passants.

Now assume that \(|S|=2q+1\). Then its complement has \(q^2-q\) points and is covered by at most \(q-1\) passants. If \(q=2\), the complement has \(2\) points, whereas a passant has \(3\), which is impossible. If \(q=3\), the complement has \(6\) points; one passant covers only \(4\) points, while two distinct passants have union of size \(7\), so this case is also impossible.

Now let \(q>4\). Since \(q-2\) lines cover at most \((q-2)(q+1)<q^2-q\) points, there must be \(q-1\) passants. Let \(m\) be the maximum number of these passants through a point. They cannot all be concurrent, so \(2\leq m\leq q-2\); the \(m\) lines through this point cover \(1+mq\) points, and each remaining line adds at most \(q+1-m\) new points. Hence their union has size at most
\[
1+mq+(q-1-m)(q+1-m)
=q^2-m(q-m)
\leq q^2-2q+4
<q^2-q,
\]
a contradiction. Thus \(q=4\) is the only remaining case. Here three passants are necessary and they must be non-concurrent; their union is precisely the complement of \(S\).
\end{proof}

\begin{theorem}
\label{thm:not bl and non r mod p}
Color the points of the plane $\PG (2,q)$ either red or blue. A line is balanced if it contains $r$ red and $b$ blue points, where $0 < r \leq b < q+1$. Assume that there are at most $2q+1$ unbalanced lines, and they do not cover all the points of $\PG (2,q)$. Furthermore, suppose that there exists an unbalanced line with not $r$ mod $p$ red points on it. Assume moreover that $q>3$. Then, the red points form one of the following structures: either a line with $q$ red points and a point outside the line; or two completely red lines and a point outside these lines. 
\end{theorem}

\begin{proof}
We may assume that we are not in the case where a line is completely red with the points outside being blue, since in this case all lines contain 1 mod $p$ red points. By Proposition \ref{prop: uncovered points}, Proposition \ref{prop: structure}, and Lemma \ref{lem: 2 lines}, applied in the dual plane, either the unbalanced lines form the symmetric difference of two pencils, or $q=4$ and they are precisely the nine lines avoiding three non-collinear points $A,B,C$.

We first exclude the latter case. Let $\ell$ be an unbalanced line whose number of red points is not congruent to $r$ modulo 2. Choose coordinates such that
\[
A=(1:0:0),\quad B=(0:1:0),\quad C=(0:0:1),
\qquad \ell:\ x+y+z=0.
\]
Let $\omega\in\mathbb{F}_4$ satisfy $\omega^2+\omega+1=0$.
The six lines
\[
x+y+z=0,\quad z=0,\quad y+z=0,\quad x+z=0,
\quad x+\omega y=0,\quad x+\omega^2y=0
\]
form a dual hyperoval. Apart from $\ell$, each passes through
one of $A,B,C$ and is therefore balanced. Counting incidences between red points and the six lines of
the dual hyperoval, each red point is counted either zero or
twice. On the other hand, $\ell$ contributes $|\ell\cap R|$
incidences and each of the other five lines contributes $r$,
since it is balanced. Hence,
\[
|\ell\cap R|+5r\equiv0\pmod2.
\]
Hence $|\ell\cap R|\equiv r\pmod2$, a contradiction.

Thus there are two points, $P_1$ and $P_2$, such that every
unbalanced line passes through one of them. Moreover, the
line joining them, $P_1P_2$, is balanced.

Assume there is a line through $P_1$, different from $P_1P_2$, that contains both a red point and a blue point (different from $P_1$). Let us denote these points as $R_1$ and $B_1$, respectively. Counting the red points through $R_1$ and $B_1$ immediately shows that the line $R_1P_2$ must be completely red (except possibly at $P_2$), and similarly, the line $B_1P_2$ must be entirely blue. Therefore, the points on the lines through $P_2$, except for $P_2$ itself, must be monochromatic. If there is no such line through $P_1$, we immediately can conclude the same for the lines through $P_1$.

Thus, for one of \(P_1,P_2\), say \(P_2\), the \(q\) lines through \(P_2\) different from \(P_1P_2\) are monochromatic apart from \(P_2\). Let \(k\) of them be red. Every line through neither \(P_1\) nor \(P_2\) is balanced and contains \(k\) red points from these lines; hence all points of \(P_1P_2\setminus\{P_1,P_2\}\) have the same color. They cannot be red, since \(r\leq(q+1)/2<q-1\) for \(q>3\), so they are blue and \(k=r\). Since \(P_1P_2\) is balanced and \(r>0\), it follows that \(r=1\) or \(2\), according to how many of \(P_1,P_2\) are red. These two cases give the two structures in the statement (the completely red line case having been excluded at the beginning).
\end{proof}

\section{The original Erd\H os problem in PG$(2,q)$ and nearby cases}
\label{sec:orig}

In this section, we consider the original problem of Erd\H os, that is $r=b=(q+1)/2$, and extend the argument to
\[
r+b=q+1,\qquad
\left|r-\frac{q+1}{2}\right|\leq\frac{p-3}{2},
\]
where $q=p^h$ and $p$ is odd. Theorem \ref{thm: unbalanced are bl set} settles this problem when the unbalanced lines cover all points of the plane. Hence we will assume that the unbalanced lines do not cover all points of $\PG(2,q)$.

For the lemmas and proposition below, we may interchange the colors and suppose that a point not covered by unbalanced lines is red; we do not assume $r\leq b$. This means that
\[
|R|=(q+1)(r-1)+1,\qquad |B|=(q+1)b.
\]
If $q>3$ and $\min\{r,b\}\geq3$, Theorem \ref{thm:not bl and non r mod p}, applied with the colors interchanged if necessary, implies that a coloring with at most $2q+1$ unbalanced lines must have $r$ mod $p$ red points on every line. Indeed, $q>4$, and the two exceptional structures have $\min\{r,b\}\leq2$. If $q=p$, our bounds on $r$ imply that every line with $r$ mod $p$
red points is balanced. Thus every line would be balanced,
a contradiction. Hence $q\geq p^2$.

\begin{lemma}
\label{lem:discr}
Let $r+b=q+1$ and assume that the unbalanced lines, $U$, do not form a dual blocking set. Let $R$ be the set of red points and $B$ the set of blue points. Then
\[
D:=\sum_{\ell\in U}(|\ell\cap R|-r)^2
 =\sum_{\ell\in U}(|\ell\cap B|-b)^2=qrb.
\]
\end{lemma}

\begin{proof}
Let $r_i$ denote the number of red points on the line $\ell_i$.
The standard equations give
\[
\sum_i r_i=|R|(q+1),\qquad
\sum_i r_i^2=|R|(|R|+q).
\]
Since balanced lines contribute zero, we obtain
\begin{align*}
D&=\sum_i(r_i-r)^2\\
 &=|R|(|R|+q)-2r|R|(q+1)+(q^2+q+1)r^2\\
 &=qrb,
\end{align*}
where we used $|R|=(q+1)r-q$ and $r+b=q+1$.
Finally, $|\ell\cap R|-r=-(|\ell\cap B|-b)$.
\end{proof}

We will call a line a blue (or red) {\em majority} line if it contains more than $b$ blue (or $r$ red) points.

\begin{lemma}
Let $r+b=q+1$ and $|r-(q+1)/2|\leq(p-3)/2$. Assume that the unbalanced lines do not form a dual blocking set and that every line contains $r$ mod $p$ red points. Then at least three blue majority lines pass through every blue point.
\end{lemma}

\begin{proof}
Counting the blue points on the lines through a blue point $P$, we count every other blue point once and $P$ itself $q+1$ times. Subtracting $b$ for each of the $q+1$ lines gives
\[
\sum_{\ell\ni P}(|\ell\cap B|-b)
=|B|+q-(q+1)b=q.
\]
Write $r=(q+1)/2+e$, where $|e|\leq(p-3)/2$. By the congruence assumption, every line contains at least $(p+1)/2+e$ red points, and hence at most $q+1-(p+1)/2-e$ blue points. Thus each positive summand is at most $(q-p)/2$. Two positive summands cannot give a total of $q$, and hence at least three blue majority lines pass through $P$.
\end{proof}

\begin{proposition}
Let $r+b=q+1$ and $|r-(q+1)/2|\leq(p-3)/2$. Assume that the unbalanced lines do not form a dual blocking set and that every line contains $r$ mod $p$ red points. Then the number of unbalanced lines is at least $2q+2$.
\end{proposition}

\begin{proof}
Suppose, for a contradiction, that there are at most $2q+1$ unbalanced lines. As noted at the beginning of the section,
$q\geq p^2$. By the preceding lemma, at least three blue
majority lines pass through every blue point.

Put $r=(q+1)/2+e$, where $|e|\leq(p-3)/2$. Let $\ell$ be the line containing the maximum number of blue points, and suppose that $\ell$ has $d$ red points and $q+1-d$ blue points. The congruence assumption gives
\[
d\geq\frac{p+1}{2}+e\geq2.
\]
Since $\ell$ is a blue majority line, $r-d$ is a positive multiple of $p$, so $d\leq r-p<(q+1)/2$.

We now show that every line contains at most $2r-d$ red points. Counting through a red point $P$ gives
\[
\sum_{m\ni P}(|m\cap B|-b)=0,
\]
so a red point lies either on zero or on at least two unbalanced lines. Let $\ell'$ be a red majority line. If a red point on $\ell'$ lies on exactly two unbalanced lines,
say $\ell'$ and $\ell''$, then the deviations from $r$ on these
two lines sum to zero. Since $|\ell''\cap B|\leq q+1-d$ by the
choice of $\ell$, we obtain
\[
|\ell'\cap R|-r
=r-|\ell''\cap R|
=|\ell''\cap B|-b
\leq r-d.
\]
If there is no such red point on $\ell'$, then through each red point on $\ell'$ there must pass at least two unbalanced lines other than $\ell'$. Moreover, through every blue point on $\ell'$ there are at least three blue majority lines, all different from $\ell'$. Thus every point of $\ell'$ is on at least two unbalanced lines other than $\ell'$; so we would see at least $2(q+1)+1$ unbalanced lines, a contradiction. In particular, every term in $D$ is at most $(r-d)^2$.

We now provide an upper bound on the sum $D$ from Lemma~\ref{lem:discr}. To do so, we first show that the number of blue points on $\ell$ with exactly three unbalanced lines is at least $q+2-2d$. Assume, for a contradiction, that there are at least $d$ blue points on $\ell$ with more than three unbalanced lines. Each of the $d$ red points on $\ell$ lies on at least one further unbalanced line, $d$ of the blue points lie on at least three further unbalanced lines each, and the remaining $q+1-2d$ blue points lie on at least two each. These lines are distinct, so, including $\ell$, we obtain at least
\[
1+d\cdot1+d\cdot3+(q+1-2d)\cdot2=2q+3,
\]
which is again a contradiction.

Let $T$ be a blue point on $\ell$ with exactly three unbalanced lines: $\ell$, $\ell_1$, and $\ell_2$. All three are blue majority lines. Since
\[
(|\ell_1\cap B|-b)+(|\ell_2\cap B|-b)=q-r+d
\]
and each summand is at most $r-d$, their sum of squares is maximized when one equals $r-d$ and the other equals $q-2r+2d$. Hence,
\[
(|\ell_1\cap B|-b)^2+(|\ell_2\cap B|-b)^2
\leq(r-d)^2+(q-2r+2d)^2.
\]
Choose $q+2-2d$ such points $T$; they give disjoint pairs of terms in $D$. The remaining terms can be estimated from above by $(r-d)^2$. Thus, using the fact that there are at most $2q+1$ unbalanced
lines, we obtain
\[
D\leq(q+2-2d)(q-2r+2d)^2
+(q-1+2d)(r-d)^2.
\]

Put $t=d-e$ and $K=(p-3)/2$. The congruence assumption gives $t\geq(p+1)/2$. The three blue majority lines through $T$ also give $q\leq3(r-d)$, so $6t\leq q+3$. Using $D=qrb$ and multiplying the preceding inequality by four, we obtain $0\leq f_e(t)$, where
\begin{align*}
f_e(t)
&=4(q+2-2t-2e)(2t-1)^2 +(q-1+2t+2e)(q+1-2t)^2
-q\bigl((q+1)^2-4e^2\bigr).
\end{align*}

We show that $f_e(t)<0$. First,
\begin{align*}
f_0(t)
&=-2t(q-3t+3)^2-(q-1)^2 -2(t-2)(3t^2-2t+4)-8<0,
\end{align*}
since $t\geq(p+1)/2\geq2$ and so each term is negative. This settles the case $p=3$, since then $e=0$.

Now let $p\geq5$. Since $6t\leq q+3$, we have
\[
A:=(q+1-2t)^2-4(2t-1)^2\geq0.
\]
As $|e|\leq K$, it follows that
\[
f_e(t)=f_0(t)+2eA+4qe^2
\leq f_0(t)+2KA+4qK^2=f_K(t).
\]
Moreover,
\begin{align*}
f_K(t+1)-f_K(t)
&=-72\left(t-\frac{3q+4-6K}{18}\right)^2 +16\left(\frac13-K\right)q
+8\left(K-\frac23\right)^2-6<0,
\end{align*}
since $K\geq1$ and $q\geq p^2=(2K+3)^2$. Therefore,
\begin{align*}
f_e(t)
&\leq f_K\left(\frac{p+1}{2}\right)\\
&=-5q^2+(2p^2+8)q-6p^3+13p^2\\
&=-5(q-p^2)^2-8(p^2-1)(q-p^2) -3p^2(p^2+2p-7)<0.
\end{align*}
This is a contradiction.
\end{proof}

Theorem \ref{thm:not bl and non r mod p} together with the preceding proposition yield immediately the following corollary. Its assumptions do not prescribe the color of the points not covered by unbalanced lines.

\begin{corollary}
\label{cor:nearby}
Let $q=p^h>3$, where $p$ is odd, and let $r+b=q+1$,
\[
3\leq r\leq b,\qquad r\geq\frac{q-p+4}{2}.
\]
If the unbalanced lines do not form a dual blocking set, then there are at least $2q+2$ unbalanced lines.
\qed
\end{corollary}

Applying the preceding corollary with $r=b=(q+1)/2$ and combining it with Theorem~\ref{thm: unbalanced are bl set}, noting that in the latter only cases (3) and (4) are possible, immediately gives the following theorem.

\begin{theorem}
\label{thm:erdos}
Let $r=b=\frac{q+1}{2}$ and $q>3$. Then the number of unbalanced lines is at least $2q+1$ unless the coloring is the trivial coloring defined in Definition \ref{def: trivial}. \qed
\end{theorem}

\section{The case of the Desarguesian plane of prime order}
\label{sec: prime}
Let us consider the pointset $R$ that is the points colored red. 
We have already noted that the case when $q = p$ is prime is special in the sense that, in this case most of the time, the set of non-$(r \bmod p)$-secants of $R$ coincides with the set of unbalanced lines. Under the assumption $r \leq b$, there are two exceptions, namely when $r=0$ a line with $p$ red points is unbalanced but contains $r$ mod $p$ red points, similarly when $r=1$ a line with $p+1$ red points is unbalanced but contains $r$ mod $p$ red points.

First, we will consider the case when the number of unbalanced lines ($\delta$) is relatively small (at most $\max\{3p+1,4p-22\}$). Subsequently, we will allow $\delta$ to be at most $d\sqrt{p}(p+1)$, while ensuring that the number of red points $r$ (or the number of blue points $b$) on a balanced line, remains small, at most $c\sqrt{p}$. For more on the constants $c, d$ see Theorem \ref{thm:mainprime}

We will consider multisets of $\PG (2,q)$, $q=p^h$, $p$ prime,  these sets are related to codewords generated by lines of the dual plane  of $\PG (2,q)$. Such codes are naturally considered over the prime field of $F_p$ and so the multiplicities of the points in a multiset will also be considered over $F_p$. We mostly consider the multiplicities as integers between $0$ and $p-1$, but there will be cases when we want to minimize their absolute values and so we will allow negative multiplicities. We will say that a point is not in a given multiset if it has multiplicity $0$. The size of a multiset is the sum of the multiplicities and the intersection number of a line and a multiset is the sum of the multiplicities of the points on that line. 

\subsection{Not too many unbalanced lines}
\label{subs: not too many}
\subsubsection{A very special case}
From another aspect, the very special case of the Erd\H{o}s problem in question has already been studied by Kiss and Somlai. They investigated point sets in $\mathrm{AG}(2,p)$. A point set $M \subset \mathrm{AG}(2,p)$ is said to admit a {\em non-special} direction $(d)$ if every line with slope $d$ contains exactly $r$ points, for some fixed integer $1 \leq r \leq p$. A direction is {\em special} otherwise. They characterized point sets admitting at most three special directions.

Color the points of $M$ red and all other points in $\mathrm{PG}(2,p)$ blue, and call a line balanced if it contains exactly $r$ red points. Then, all lines in a non-special direction are balanced, and hence a point set with few special directions results in a coloring with few unbalanced lines.
 
Given a prime number $p$, we can interpret the elements of $\mathbb{F}_p$ as the integers from $0$ to $p-1$. This naturally defines an ordering $<$ on the elements of $\mathbb{F}_p$. Kiss and Somlai proved the following results. 

\begin{example}[Kiss, Somlai] \cite{KS}
 \label{Ex:KissSomlaiExample}
 Let $p > 2$ be a prime. Then the set
 \[
  S = \sett{(x,y) \in \mathbb{F}_p^2}{ y < x }
 \]
 has exactly three special directions, namely $(0)$, $(1)$ and $(\infty)$.
\end{example}

\begin{result}[Kiss, Somlai] \cite{KS}
 \label{Res:KissSomlai}
Let $M$ be a point set in $\mathrm{AG}(2,p)$. Let $(d)$ be a non-special direction and assume that on each line with slope $d$, we see $1 \leq r \leq p$ points. If $M$ admits at most three special directions, then $M$ is either a union of $r$ parallel lines or, up to an affine transformation, it is equivalent to $S$ or its complement in Example~\ref{Ex:KissSomlaiExample}.
\end{result}

The above result can be translated into the following very special case of our Erd\H{o}s problem. Although it may seem extreme, we will see later that it will be very useful.

\begin{proposition}
\label{prop:affinecase}
Color the points of $\mathrm{PG}(2,p)$, where $p \geq 19$, into red and blue. A line is called balanced if it contains exactly $r$, where $0<r<p+1$, red points. Suppose that there are at most $\max\{3p+1, 4p-22\}$ unbalanced lines and 
assume that there is a line $\ell$ so that all unbalanced lines different from $\ell$ intersect $\ell$ in at most three different points, $P_1, \dots, P_k$, $1\leq k \leq 3$.
Then one of the following holds: \begin{itemize} \item[(1)] The red point set is the union of \(r\) concurrent lines through one of the points \(P_i\), with their common point removed, together with a subset of \(\{P_1,\dots,P_k\}\); and the blue point set has the analogous form with \(r\) replaced by \(b=p+1-r\). \item[(2)] One of the two color classes is, up to an affine transformation, equivalent to \(S\) or its complement in Example~\ref{Ex:KissSomlaiExample}, together with a subset of the directions \(\{(0),(1),(\infty)\}\). \end{itemize}
\end{proposition}

\begin{proof}
Let $P$ and $Q$ be two points on a line $\ell$ through which all lines except $\ell$ are balanced. By counting the number of red points on the lines through $P$ and then through $Q$, we deduce that $P$ and $Q$ must have the same color, say blue. 
Now, let $\ell$ be the line at infinity and consider the set of red points outside $\ell$. If there were none, then $r$ must have been $0$, which is not possible. Otherwise the result follows from Result \ref{Res:KissSomlai}.
\end{proof} 

\subsubsection{Preliminary results on codewords and special pointsets}
In this subsection, similarly to \cite{ASzW}, we will establish a connection between the set of red points of $\PG (2,p)$ and a codeword of the code $C_1(2,p)$. The results will then follow from earlier results on small weight codewords. 

\begin{definition}
 Let $\mathcal{P} = \{ P_1, P_2, \dots, P_{p^2+p+1} \}$ denote the set of points of ${\rm PG}(2,p)$, $p$ is prime. We define the \emph{characteristic vector} of a line $\ell$ in ${\rm PG}(2,p)$ as the vector 
\[
\chi_{\ell} = (\chi_{\ell,1}, \chi_{\ell,2}, \dots, \chi_{\ell,p^2+p+1}) \in \F_p^{p^2+p+1},
\]
where
\[
  \chi_{\ell,i} =
  \begin{cases}
   1 & \text{if } P_i \in \ell, \\
   0 & \text{otherwise}.
  \end{cases}
\]
The code $C_1(2,p)$ is defined as the $\F_p$-span of the characteristic vectors of the lines of ${\rm PG}(2,p)$. 

Let $c = (c_1, c_2, \dots, c_{p^2+p+1})$ be a codeword in $C_1(2,p)$. Then, the weight of the codeword, denoted by $w(c)$, is the number of non-zero coordinates in $c$. The support, $supp(c)$, of $c$ is the set of points $P_i$ corresponding to the non-zero coordinates of $c$.
\end{definition}

We will use the following result from the last two authors. For the exact statement, see \cite[Theorem 4.8, Corollary 4.10]{SzW}; here, we only require a simplified version.

\begin{result}[Sz\H onyi, Weiner] \cite{SzW}
 \label{Res:SzWPrime}
 Suppose that $p \geq 19$ is a prime and that $c \in C_1(2,p)$ with $w(c) \leq \max\{3p+1, 4p-22\}$.
 Then $c$ is either a linear combination of at most three lines or 
$supp(c)$ can be covered by three concurrent lines.
\end{result}

\begin{notation}
Coordinates of a point in ${\rm PG}(2,p)$ will be denoted by $(a, b, c)$ while the coordinates of a line with $[a, b, c]$.
\end{notation}

Next we will establish the connection between the coloring of the points of ${\rm PG}(2,p)$ and codewords.

\begin{definition}
\label{def: codeword of R}
Color the points of $\rm{PG}(2, p)$ red and blue and consider the set of red points $R$. Define  
\[
  c_R = \sum_{(x,y,z) \in R} \chi_{[x,y,z]}
\]
\end{definition}

Note that the coordinate of $c_R$ corresponding to the point $(a, b, c)$ is the number of points of $R$ on the line $[a, b, c]$ modulo p.

\begin{lemma}
\label{lem:codeword connection}   
Let $\underline{j}$ be the all $1$ vector. The coordinate of the codeword $c_R - r\underline{j}$ corresponding to the point $(a, b, c)$ is non-zero if and only if the line $[a, b, c]$ intersects the point set $R$ in not $r$ mod $p$ points. \qed
\end{lemma}

In our case, $w(c_R - r\underline{j})$ will be small and hence we will be able to apply Result \ref{Res:SzWPrime}.

\subsection{Colorings with at most $\max\{3p+1, 4p-22\}$ unbalanced lines}
\label{subs: st most max}
In this subsection, we will assume that $p \geq 19$ prime. We color the points of $\PG (2,p)$ red and blue. We call a line balanced if it contains exactly $r$ red points, where $0\leq r\leq b=p+1-r$. Throughout this subsection, we will assume that the number of unbalanced lines is at most $\max\{3p+1, 4p-22\}$.
As mentioned before, unbalanced lines which contain not $r$ mod $p$ red points will be handled by earlier results.

\begin{definition}
We call an unbalanced line containing not 
$r$ mod $p$ red points a {\rm regular unbalanced} line. 
\end{definition}
Under the assumption \(0\leq r\leq b\), a non-regular unbalanced line can occur only when \(r=0\) or \(r=1\): in the former case it is a \(p\)-secant of \(R\), while in the latter case it is a completely red line. 
 We first examine these two exceptional cases.

\begin{lemma}
\label{lem:exceptional}
Let $R$ be the set of red points and let $\delta$ denote the number of unbalanced lines. Then the following hold:
\begin{itemize}
    \item[(1)] If \(r=0\),  for every integer \(k\) with
\(
1\leq k\leq \min\{|R|,p+1\},
\)
we have
\[
\delta\geq k(p+1)-\binom{k}{2}.
\]
    Moreover, if there is a non-regular unbalanced line, then $\delta\geq p^2+1$.
    \item[(2)] If $r=1$ and $\ell$ is a completely red line, then for every integer \(k\) with
\(
1\leq k\leq \min\{|R\setminus\ell|,p+1\},
\)
we have
\[
\delta\geq 1+k(p+1)-\binom{k}{2}.
\]
    \item[(3)] If \(2\leq r\leq b\), then every unbalanced line is regular.
\end{itemize}
\end{lemma}

\begin{proof}
For any $k \leq p+1$ distinct points, the union of the pencils through them contains at least
\[
k(p+1)-\binom{k}{2}
\]
lines: when the $i$-th pencil is added, at most $i-1$ of its lines have occurred before.

If $r=0$, every line through a red point is unbalanced, which gives the first inequality in (1). A non-regular unbalanced line then contains exactly $p$ red points. The $p$ red points on such a line determine $p^2$ distinct lines different from it, all unbalanced, giving at least $p^2+1$ unbalanced lines.

If $r=1$ and $\ell$ is completely red, every line through a red point outside $\ell$ contains at least two red points and hence is unbalanced. The line $\ell$ itself is not among these pencils, which proves (2).

Finally, (3) is immediate.
\end{proof}

\begin{lemma}
\label{lem: at most 3 pts}
Color the points of $\mathrm{PG}(2,p)$, where $p \geq 19$ is prime, red and blue, and call a line balanced if it contains exactly $0\leq r\leq b=p+1-r$ red points. Assume that the number of unbalanced lines is at most $\max\{3p+1, 4p-22\}$. Then, the following holds:
\begin{itemize}
    \item[(1)] There exist at most three points with the property that every regular unbalanced line contains at least one of them.
    \item[(2)] Suppose that we need three non-collinear points, $P_1$, $P_2$, and $P_3$, in order to block all regular unbalanced lines.
    Then there exist non-zero multiplicities $m_1, m_2, m_3 \in F^*_p$ so that the union of the points $P_i$ with multiplicity $m_i$ and the set of red-colored points forms a multiset in which every line contains exactly $r$ mod $p$ points (counted with multiplicities).
\end{itemize}
\end{lemma}

\begin{proof}
For the set of red points $R$, construct the codeword $c_R$ (see Definition~\ref{def: codeword of R}), and define
\[
c^*_R = c_R - r\underline{j},
\]
where $\underline{j}$ is the all-ones vector. By Lemma \ref{lem:codeword connection} and by the assumption of the lemma, we have 
\[
w(c^*_R) \leq \max\{3p+1, 4p-22\}.
\]
Thus, by Result \ref{Res:SzWPrime}, the support of $c^*_R$ can be covered by at most three lines, which immediately implies $(1)$.

Now, assume that exactly three lines $\ell_1, \ell_2, \ell_3$ are required to cover $supp(c^*_R)$ and that these lines are not concurrent. By Result \ref{Res:SzWPrime}, there exist non-zero multiplicities $m^*_1, m^*_2, m^*_3 \in \mathbb{F}^*_p$ such that
\[
c^*_R = \sum_{i=1}^3 m^*_i\chi_{\ell_i}.
\]
Rearranging, we obtain
\[
r\underline{j} = c_R - \sum_{i=1}^3 m^*_i\chi_{\ell_i} = c_R + \sum_{i=1}^3 (p - m^*_i)\chi_{\ell_i}.
\]
and hence $(2)$ follows.
\end{proof}

\begin{proposition}
\label{prop:noncollin3pts}
Suppose that $p>5$ and assume that there exist three non-collinear points $P_1$, $P_2$, $P_3$, so that each regular unbalanced line contains at least one of them. Suppose that there are at most $\max\{3p+1,4p-22\}$ unbalanced lines. Suppose moreover that there exists a multiset $M$ which agrees with $R$ outside $\{P_1,P_2,P_3\}$ and in which every line contains exactly $r$ mod $p$ points (counted with multiplicities).
Then $R\setminus \{P_1, P_2, P_3\}$ is the union of \(r\) lines, each passing through one of the points \(P_i\), with \(P_1,P_2,P_3\) removed; or $r=1$ and $R$ consists of a completely red line together with a subset of $\{P_1,P_2,P_3\}$ outside this line. 
\end{proposition}

\begin{proof}
Let $R^* = R\setminus \{P_1, P_2, P_3\}$.
Let $m_i$ denote the multiplicity of $P_i$ in $M$. First we get rid of the exceptional cases.

If $r=0$, then by Lemma~\ref{lem:exceptional}(1), there is no non-regular unbalanced line. If there were a red point $P\notin\{P_1,P_2,P_3\}$, all $p+1$ lines through $P$ would be regular unbalanced, but at most three of them could contain one of the points $P_i$, a contradiction. Thus $R^*=\emptyset$, giving the first conclusion.

If $r=1$ and there is a completely red line $\ell$, Lemma~\ref{lem:exceptional} shows that there are at most three red points outside $\ell$. In particular, $\ell$ is the unique completely red line. If a red point $P$ outside $\ell$ were not one of the $P_i$, then all $p+1$ lines through $P$ would be regular unbalanced and would have to contain one of the three points $P_i$, again a contradiction. Hence the second conclusion holds.

We may therefore assume that $r>0$ and that every unbalanced line is regular.

First, we will show that through one of the points $P_i$, $1 \leq i \leq 3$, say through $P_j$, there passes a line $\ell$ such that either all points of $\ell \setminus \{P_j\}$ are in $R^*$ or none.
Let $P$ be a point not in $R^*$ and not on the lines $P_iP_j$, for $1 \leq i \neq j \leq 3$. 
Such a point exists, since otherwise a line avoiding $P_1,P_2,P_3$ would contain at least $p-2 > (p+1)/2 \geq r$ (when $p>5$) red points and hence would be a regular unbalanced line avoiding all three points.
Also, let $Q$ be a point of $R^*$ not on the lines $P_iP_j$.
If no such point exists, then the lines through the points $P_i$ show that each side of the triangle without the vertices is either contained in $R^*$ or disjoint from $R^*$. Thus $R^*$ is the union of $t$ sides without vertices. A line avoiding $P_1,P_2,P_3$ is balanced, so $t=r$, and the first conclusion follows. Now, consider counting the points of $R^*$ once on the lines through $P$ and also on the lines through $Q$. Note that a line not passing through any $P_i$ is balanced and hence contains exactly $r$ points of $R$ and hence of $R^*$ too, and so we have

\begin{align}
(p-2)r &+ |PP_1 \cap R^*| + |PP_2 \cap R^*| + |PP_3 \cap R^*| \notag \\
&= 1 + (p-2)(r-1) + |QP_1 \cap R^*| - 1 + |QP_2 \cap R^*| - 1 + |QP_3 \cap R^*| - 1.
\end{align}

From this, we obtain:

\begin{equation}
\label{eq:3pts}
p + |PP_1 \cap R^*| + |PP_2 \cap R^*| + |PP_3 \cap R^*| = |QP_1 \cap R^*| + |QP_2 \cap R^*| + |QP_3 \cap R^*|.
\end{equation}

The multiset $M$ intersects each line in $r \pmod{p}$ points, and hence:

\[
|PP_i \cap R^*| \equiv |QP_i \cap R^*| \pmod{p}.
\]

Since $R^*$ is a set, equation \eqref{eq:3pts} can only be true if for some $1 \leq j \leq 3$,

\[
\left| |PP_j \cap R^*| - |QP_j \cap R^*| \right| = p,
\]

which can only happen if one of $|PP_j \cap R^*|$ or $|QP_j \cap R^*|$ is $0$ and the other is $p$. Hence, we have found the required line $\ell$ through the point $P_j$.

Without loss of generality, we may assume that through $P_1$ there passes a line $\ell$, such that $|\ell \cap R^*| = p$. Since $M$ intersects every line in $r$ mod $p$ points, and all these lines have the same contribution at $P_1$ in $M$, this implies that any line through $P_1$ and not passing through $P_2$ or $P_3$ must contain either $0$ or $p$ points from $R^*$. 

Assume that there are $k$ lines through $P_1$ containing $p$ points from $R^*$. To complete the proof, we only need to show that on the line $P_1P_2$ (and similarly on $P_1P_3$), there are either $0$ or $p-1$ points of $R^*$. 

Consider the lines through $P_3$ which do not pass through $P_2$ or $P_1$. From above, such a line intersects $R^*$ in either $k+1$ or $k$ points, depending on whether the intersection with the line $P_1P_2$ is in $R^*$ or not. 
However, since $M$ contains $r$ points modulo $p$ on each line, it follows that these lines intersect the set $R^*$ in the same number of points, namely in $(r - m_3)$ modulo $p$ points.

Therefore, all points of $P_1P_2 \setminus \{P_1, P_2\}$ belong to $R^*$ simultaneously, or none of them do. 
By interchanging the roles of $P_2$ and $P_3$, the same holds for $P_1P_3$. Thus $R^*$ is the union of $t$ lines through $P_1$, with $P_1,P_2,P_3$ removed. Any line avoiding $P_1,P_2,P_3$ is balanced and meets $R^*$ in exactly $t$ points, so $t=r$, and the proof follows.
\end{proof}

Finally, we are ready to state our main theorem of this section.

\begin{theorem}
\label{thm:prime max}
Consider a coloring of the points of ${\rm PG}(2, p)$, where $p \geq 19$, with red and blue. A line is called balanced if it contains exactly $0\leq r\leq b=p+1-r$ red points. Suppose that there are at most $\max\{3p+1, 4p-22\}$ unbalanced lines. 

Then one of the following holds:
\begin{itemize}
    \item[(1)] There is a point \(P\) such that, apart from \(P\), the points on \(r\) lines through \(P\) are red and the points on the remaining \(b=p+1-r\) lines are blue. The point \(P\) can have either color.
    \item[(2)] A coloring obtained from (1) by switching the colors of at most two points different from \(P\).
    \item[(3)] \(r=1\), and the red points consist of a line together with at most three points outside it.
    \item[(4)]  Either \(r=(p-1)/2\) and the red point set is, up to an affine transformation, equivalent to \(S\) from Example \ref{Ex:KissSomlaiExample} together with a subset of the directions \(\{(0),(1),(\infty)\}\); or \(r=b=(p+1)/2\) and one of the two color classes is equivalent to the complement of \(S\) in \(\mathrm{AG}(2,p)\), together with a subset of these three directions.
    \item[(5)] \(r=3\), and the red point set is obtained from the union of the three sides of a triangle by possibly switching the colors of its three vertices.
\end{itemize}

\end{theorem}

\begin{proof}
Consider the set of red points $R$.
First suppose that $r=0$. If $|R|\geq4$, Lemma~\ref{lem:exceptional} gives
\[
\delta\geq4(p+1)-\binom42=4p-2>\max\{3p+1,4p-22\},
\]
a contradiction. Hence $|R|\leq3$, which is contained in cases (1) and (2).

Next suppose that $r=1$ and there is a completely red line $\ell$. If there were at least four red points outside $\ell$, Lemma~\ref{lem:exceptional} would give
\[
\delta\geq1+4(p+1)-\binom42=4p-1>\max\{3p+1,4p-22\}.
\]
Thus this gives case (3). We may therefore assume that $r>0$ and that every unbalanced line is regular.

By Lemma \ref{lem: at most 3 pts} (1), the regular unbalanced lines can be covered by at most three points, say by the set of points $B$. Choose $B$ minimal with this property.
If there is a line containing all points of $B$, then Proposition \ref{prop:affinecase} completes the proof. Its first case gives (1) or (2), while its second case, using $r\leq b$, gives (4).
Otherwise, $B$ consists of three non-collinear points, in which case Proposition \ref{prop:noncollin3pts}, together with Lemma \ref{lem: at most 3 pts} $(2)$ completes the proof. The proof of Proposition \ref{prop:noncollin3pts} shows that its first conclusion gives (1) or (2), except when $r=3$ and the three lines are the sides of a triangle, which gives (5). Its second conclusion gives (3).
\end{proof}
\subsection{More unbalanced lines with some additional restrictions}
\label{subs: more unbalanced with restr}
\subsubsection{Preliminaries}

We recall the stability result of $ k \bmod p $ sets of \cite{SzW}, translated into the coloring terminology used above. The results below are for general $q$.
Throughout these preliminary results, $\delta$ denotes the number of regular unbalanced lines; this number is at most the total number of unbalanced lines.

\begin{result}[{\cite[Theorem 1.1]{SzW}}]
\label{primelargeindpts}
Color the points of $\PG(2, q)$, $q=p^h$ red and blue, $q > 17$ and assume that the number of regular unbalanced lines, $\delta$,
is smaller than $(q+1) {\sqrt \frac{q}{2}} $. Then there 
exists a point set \( L \), of size $ \lceil \frac{\delta}{q+1} \rceil $, such that every regular unbalanced line contains at least one point of 
$L$. Moreover, by \cite[Remark 3.3]{SzW}, $L$ can be chosen so that each $P\in L$ lies on at least
\[
q+1-\left(\frac{\delta}{q+1}+\frac{2\delta^2}{(q+1)^3}\right)
\]
regular unbalanced lines.
\end{result}

From this result, we get the following:

\begin{corollary}
\label{numofbadlinesthrpt}
Let $P$ be a point of $\PG(2, q)$, $ q > 17 $. 
If $P\notin L$, then at most $|L|$ regular unbalanced lines pass through $P$. In particular, if every unbalanced line is regular, then the number of balanced lines through $P$ is at least $q+1-|L|$.
\end{corollary}

\begin{property}[{\cite[Property 3.5]{SzW}}]
\label{propkmodp}
Color the points of $ \PG(2, q) $ red and blue. We say that 
the property holds if whenever more than $ q/2 $ 
regular unbalanced lines pass through a point, there exists an integer $ m $ such that on 
more than $ 2 \frac{\delta}{q+1} + 5 $ of these lines the number of red points is congruent to
$ m \bmod p $.
\end{property}

\begin{result}[{\cite[Theorem 3.6]{SzW}}]
\label{stab} 
Color the points of $ \PG(2, q) $,  $ q > 17 $, red and blue. Assume 
$ \delta < (\lfloor \sqrt q\rfloor + 1)(q + 1 - \lfloor \sqrt q\rfloor ) $. 
Assume furthermore that Property \ref{propkmodp} holds. Then there 
exists a multiset $ R^* $ such that every line intersects $ R^* $ in 
$ r \bmod p $ points (counting with multiplicity) and the number of points 
of $ \PG(2, q) $ which have different multiplicities in $R$ than in $R^*$
is $ \lceil \frac{\delta}{q+1}\rceil $.
\end{result}

\subsubsection{The number of unbalanced lines is $O(p\sqrt{p})$ when $r$ is either small or large}

In this subsection, our aim is to prove the following theorem, showing that when $r$ is small (or large) and the number of unbalanced lines is moderately bounded, the coloring is close to a trivial one.

\begin{theorem}
\label{thm:mainprime}
Color the points of $\PG(2,p)$, $p$ prime,  red and blue. We say that a line is balanced if it contains exactly $r$ red points.
Suppose that $0\leq r<c\sqrt p$ and the number of unbalanced lines, $\delta$, is less than $d\sqrt p(p+1) $, where $c, d > 0$ and $8d(c+d)  + \frac{20c+28d}{\sqrt p} + \frac{19}{p}< 1$ hold.
Then we can change the color of at most $\lceil d\sqrt{p}\rceil$ points so that the coloring we obtain 
contains $r$ completely red lines and all other points are blue. Furthermore, the size of the union of the set of recolored points and the intersection points of the $r$ red lines is at most $\lceil d\sqrt{p}\rceil$.
\end{theorem}

\begin{remark}
\label{rm: constants}
Note that the assumption 
$c, d > 0$,
and 
$8d(c+d)  + \frac{20c+28d}{\sqrt p} + \frac{19}{p}< 1$ immediately implies that $8d^2 < 1$, also $p > 19$; and hence Result \ref{stab} can be applied once Property~\ref{propkmodp} has been verified. Also note that $8cd < 1$. Furthermore, since $\frac{20c+28d}{\sqrt p} < 1$, we also see that $\frac{20c+20d}{\sqrt p} < 1$ and so $20(c+d) < \sqrt p$.
\end{remark}

\begin{remark}
Note that $c$ and $d$ do not have to be constants. For example, when $c = C\sqrt p$ and  $d = D / \sqrt p$ and $C, D$ are constants and $p$ is large enough, we need that $8CD + 20C < 1$. In this case Theorem \ref{thm:mainprime} says that there are $r$ almost completely red lines.
If these lines are not concurrent, we see immediately at least $r$ intersection points. If now $r>\lceil D\rceil$ then Theorem \ref{thm:mainprime} yields that these $r$ lines must be concurrent. Hence the coloring is the trivial coloring with changing the colors of at most $\lceil D\rceil$ points.
\end{remark}

\textbf{Assumption.} In the lemmas below we assume, in addition to the hypotheses of Theorem \ref{thm:mainprime}, that $r>0$ and every unbalanced line is regular.

The excluded cases will be handled separately in the proof of the theorem.

\begin{lemma}
\label{lem:set L}
There exists a set of points, $L$, such that each unbalanced line contains at least one point from $L$ and 
\[|L| = \lceil \frac{\delta}{p+1} \rceil \leq \lceil d\sqrt p \rceil. \]
Moreover, $L$ can be chosen so that every $P\in L$ lies on at least
\[
p+1-\left(\frac{\delta}{p+1}+\frac{2\delta^2}{(p+1)^3}\right)
\]
unbalanced lines.
\end{lemma}

\begin{lemma}
\label{numofredpts}
For the set of red points $R$ in $\PG(2,p) $ we have \[ |R| \leq (r+ |L|)(p+1) \leq 
(r+\lceil d\sqrt p \rceil)(p+1). \]
\end{lemma}

\begin{proof}
If every red point belongs to $L$, the assertion is immediate. Otherwise, we count the number of red points on the lines through a red point not in $L$. 
By Corollary \ref{numofbadlinesthrpt}, we count at most 
$ 1 + (r-1)(p+1 - |L|) +|L|p $ of them.
\end{proof}

\begin{lemma}
\label{lotsofsamethrpt}
Through any point of $ L $ there are at least
$ \frac {p+1}{4(c\sqrt p + \lceil d \sqrt p \rceil)}$ lines that contain the same number of red points. 
\end{lemma}

\begin{proof}
Fix $P\in L$, and let $x_1,\ldots,x_{p+1}$ be the numbers of red points on the lines through $P$, excluding $P$ itself. Then
\[
\sum x_i\leq |R|.
\]
Let $M$ be the maximum number of the $x_i$ having the same value. Since each integer occurs at most $M$ times, the smallest possible sum is obtained by taking $M$ zeros, then $M$ ones, and so on. Hence
\[
|R|\geq\sum x_i\geq\frac{(p+1)(p+1-M)}{2M},
\]
and therefore, by Lemma~\ref{numofredpts},
\[
M\geq\frac{(p+1)^2}{2|R|+p+1}
\geq\frac{p+1}{2(r+|L|)+1}
>
\frac{p+1}{4(c\sqrt p+\lceil d\sqrt p\rceil)}.
\]
\end{proof}

\begin{lemma}
\label{lem:rstar}
There exists a multiset of points $ R^* $, such that the set of points whose multiplicity differs in $R^*$ and $R$ is exactly $L$, and every line intersects $R^*$ in $r \bmod p$ points. 
\end{lemma}

\begin{proof}
By Remark \ref{rm: constants}, we have $d < \frac{1}{2\sqrt 2}$. 
Hence
\[
\delta<\frac{\sqrt p(p+1)}2
<(\lfloor\sqrt p\rfloor+1)(p+1-\lfloor\sqrt p\rfloor),
\]
so the numerical bound in Result \ref{stab} is satisfied.
By Lemma \ref{lotsofsamethrpt}, through each point in $L$ there are at least
\[
\frac {p+1}{4(c\sqrt p + \lceil d \sqrt p \rceil)}
\]
lines containing the same number of red points. 
It remains to verify that Property \ref{propkmodp} holds.
A point outside $L$ lies on at most $|L|<p/2$ unbalanced lines, so it is enough to consider points of $L$. For $P\in L$, put
\[
A=\frac{\delta}{p+1}+\frac{2\delta^2}{(p+1)^3}.
\]
By Lemma \ref{lem:set L}, at most $A$ balanced lines pass through $P$.

\[
 \frac {p+1}{4(c\sqrt p + \lceil d \sqrt p \rceil)}
 >
 \frac {p+1}{4(c\sqrt p + d \sqrt p + 1)}
 >
 \frac{2\delta}{p+1} + 5.
\]

This inequality follows from the assumption of Theorem \ref{thm:mainprime}. Moreover, $8d^2<1$ gives
\[
A<\frac{\delta}{p+1}+\frac14<\frac p2.
\]
Thus the common number of red points above cannot be $r$, and hence Property \ref{propkmodp} holds.

Hence by Result \ref{stab} there is a multiset $R^*$ such that it intersects 
every line in $r \bmod p$ points and the set $S$ consists of those points that have different multiplicity in $R$ than in $R^*$, and $S$ has size $\lceil \frac{\delta}{p+1}\rceil$. 
Clearly, every unbalanced line must contain a point of $S$. 
If $P\in L\setminus S$, then at most $|S|$ unbalanced lines through $P$ can meet $S$, while Lemma \ref{lem:set L} gives at least $p+1-A>|S|$ such lines. Thus $L\subseteq S$, and since $|L|=|S|$, we have $L=S$.
\end{proof}

\begin{lemma}
Let $ P $ be a point of $ L $ and let the multiplicity \( m_P \) of 
$ P $ in $ R^* $ satisfy $ 0 \leq m_P < p $. Then $ m_P \leq r $ or $ p - |L| 
\leq m_P $. 
\end{lemma}

\begin{proof}
Assume to the contrary that there exists a point $P \in L$ with 
multiplicity $r < m_P < p - |L|$ in $R^*$. Through the point $P$, there are at least $p + 2 - |L|$ lines that do not meet any point in $L \setminus \{P\}$. Each of these lines contains $r$ mod $p$ points from $R^*$ and contains exactly one point (namely $P$) at which the multiplicities in $R$ and $R^*$ may differ. Hence it contains exactly $p+r-m_P>|L|+r$ red points different from $P$. 
Therefore,
\[
|R| \geq (p + 2 - |L|)(|L| + r + 1).
\]

We now show that 
\[
(p + 2 - |L|)(|L| + r + 1) > (r + |L|)(p+1),
\]
which is a contradiction to Lemma \ref{numofredpts}. 
We rearrange the previous inequality to obtain:
\[
0 < p + r + 2 - |L|^2 - r|L|.
\]
 
The right-hand side is a quadratic expression in $|L|$, 
an upside down parabola.
Indeed, since $|L|\leq\lceil d\sqrt p\rceil<d\sqrt p+1$ and $r<c\sqrt p$,
\[
|L|^2+r|L|
<
d(c+d)p+(c+2d)\sqrt p+1
<p,
\]
where the last inequality follows from the assumption of Theorem \ref{thm:mainprime}.
This contradiction completes the proof since $0 < r$ and 
$0\leq |L|\leq\lceil d\sqrt p\rceil$.
\end{proof}

Now we change the multiplicities so that their absolute values are minimized and hence we will allow negative values as well. The multiplicity $m_P$ of a point of $ L $ in $ R^* $ will be $ -|L| \leq m_P \leq r $. The next lemma is straightforward.

\begin{lemma}
\label{largeindexedptswithmult}
The total multiplicity of the points in $ L $ in 
$ R^* $ lies in the interval $ [-|L|^2, r|L|] $. \qed
\end{lemma}

\begin{lemma}
\label{lem:r or p+r}
Each line intersects $ R^* $ in either $ r $ or $ p+r $ points. 
\end{lemma}

\begin{proof}
Since every line intersects $R^*$ in $r \pmod{p}$ points, we have $|\ell \cap R^*| = kp + r$ for some integer $k$.
Also, $0 \leq |\ell \cap R| \leq p+1$ and hence by Lemma \ref{largeindexedptswithmult} $-|L|^2 \leq |\ell \cap R^*| \leq p+1 + r |L|$. It follows from Lemma \ref{lem:set L} and Remark \ref{rm: constants} that $k=0, 1$.
\end{proof}

\begin{lemma}
\label{lem: pt with no p+r-secant}
There exists a blue point outside $ R^* $ (i.e. has multiplicity $0$ in 
$ R^* $) and no $(p+r)$-secant of $ R^* $ passes through it. 
\end{lemma}

\begin{proof}
Let $\ell$ be a balanced line. Then $\ell$ contains at least 
$p+1-r-|L|>p-(c+d)\sqrt p$
blue points which are not in $R^*$. 
By Remark \ref{rm: constants}, this number is larger than $19p/20$. Moreover,
\[
|\ell\cap R^*|\leq r+r|L|<p+r,
\]
where $r|L|<c\sqrt p(d\sqrt p+1)<p$. Hence $\ell$ itself is not a $(p+r)$-secant of $R^*$.

Assume to the contrary that each of them is contained in a $(p+r)$-secant of $R^*$. 
These secants are distinct, since a line different from $\ell$ meets $\ell$ in only one point.

Since $m_P\leq r$ for every $P\in L$, a $(p+r)$-secant of $R^*$ contains at least
\[
p+r-r|L|=p-r(|L|-1)>p-cdp=p(1-cd)
\]
red points.
By Remark \ref{rm: constants}, $cd<1/8$ and hence
$p(1-cd)>7p/8$.

Counting the number of red points on the $(p+r)$-secants through the blue points on $\ell$ which are not in $R^*$, 
put $s=\lfloor7p/8\rfloor$. We may choose $s+1$ of these secants, each containing at least $s+1$ red points. The $j$-th secant contributes at least $s+2-j$ new red points, and hence
\[
|R|\geq\frac{(s+1)(s+2)}2.
\]

By Lemma \ref{numofredpts}, the number of red points is less than
\[
((c+d)\sqrt p+1)(p+1)<(p/20+1)(p+1).
\]
This yields a contradiction when $p>19$ and hence the lemma follows.
\end{proof}

\noindent {\bf Proof of Theorem \ref{thm:mainprime}.} Consider the points colored red; they form the set $R$.

Put $m=\lceil d\sqrt p\rceil$. By Remark \ref{rm: constants}, $m+1<\sqrt p$, and hence
\[
\binom{m+1}{2}<p+1.
\]
Since $m+1-d\sqrt p\geq1$, Lemma \ref{lem:exceptional}, applied with $k=m+1$, shows that if $r=0$, then $|R|\leq m$, while if $r=1$ and $\ell$ is a completely red line, then $|R\setminus\ell|\leq m$. Recoloring these points blue proves the theorem in these cases. Since $r<c\sqrt p<p$, we may therefore assume that $r>0$ and that every unbalanced line is regular.

We construct the multiset $R^*$ using Lemma \ref{lem:rstar}.

We will show that $R^*$ is the weighted sum of $r$ lines. 
Counting the points of $R^*$ with 
multiplicity on the $p+1$ lines through the point provided by Lemma \ref{lem: pt with no p+r-secant} and using Lemma \ref{lem:r or p+r}, we obtain $|R^*|=(p+1)r$. Now counting the points of $R^*$ with multiplicity on the lines through an arbitrary point $P$, we get that the multiplicity of $P$ in $R^*$ must be the number of $(p+r)$-secants through $P$; hence $R^*$ is the weighted sum of the not necessarily distinct $r$ lines $e_1,e_2,\dots,e_r$.

Note that the multiplicities of the points in $R$ and $R^*$ differ only at points of $L$. This means that any point whose color was flipped must be in $L$.
Also, points that lie on at least two of the lines $e_i$ must be in $L$ since their multiplicity in $R^*$ is larger than $1$. Since $|L|<p+1$, we see that the lines $e_i$ are all different and the intersection point of any two of these is in $L$.

Now consider a point of $L$ which lies on at most one of the lines $e_i$. Since $R^*$ is the weighted sum of lines, this means that its multiplicity is either $0$ or $1$ in $R^*$. And so its multiplicity in $R$ is $1$ in the first case and $0$ in the second; hence it is a point whose color was flipped. Therefore, $L$ is exactly the union of the recolored points and the intersection points of the lines $e_i$. The result follows from Lemma \ref{lem:set L}.

\section*{Acknowledgement}
Project no.\ 153080 has been implemented with the support provided by the
Ministry of Culture and Innovation of Hungary from the National Research,
Development and Innovation Fund, financed under the ADVANCED\_25 funding scheme.

The third author was partially supported by the Slovenian Research Agency,
research project N1-0429 and by Dynasnet European Research Council Synergy project ERC-2018-SYG 810115.

The research presented in this paper was carried out by the authors several
years ago, without the use of any AI tool; essentially all of the mathematical
content is their own work. The present version was proofread with the help of
ChatGPT. During this proofreading, ChatGPT pointed out that the original statement of Lemma~\ref{lem: 2 lines} fails for $q=4$; the lemma and its proof were corrected accordingly. The details of the extension to nearby cases in Section~\ref{sec:orig} were also adjusted at this stage.

\end{document}